\documentclass[12pt]{article}
\usepackage[sc]{mathpazo}
\usepackage[T1]{fontenc}
\usepackage[utf8]{inputenc}
\usepackage{microtype}
\usepackage{geometry}
\usepackage{titlesec}
\titleformat{\subsection}[runin]{\bf}{\thesubsection.}{3pt}{}
\usepackage{amsmath, amsfonts, amssymb, tikz-cd, mathtools}
\usepackage{graphicx}
\usepackage{caption, subcaption}
\usepackage{enumitem}
\DeclareMathOperator{\MCG}{\mathcal{MCG}}
\DeclareMathOperator{\PMCG}{\mathcal{PMCG}}
\DeclareMathOperator{\End}{End}
\DeclareMathOperator{\Id}{Id}
\usepackage{amsthm}
\newtheorem{thm}{Theorem}
\newtheorem{lem}{Lemma}[section]

\newtheorem{prop}[lem]{Proposition}
\newtheorem{ques}[lem]{Question}
\theoremstyle{remark}
\newtheorem{rem}{Remark}[section]

\theoremstyle{definition}

\usepackage{bbm}
\usepackage{hyperref}
\definecolor{upsPurp}{RGB}{99,0,60}
\definecolor{upsRed}{RGB}{198,11,70}
\definecolor{upsCyan}{RGB}{0,148,181}
\hypersetup{
     colorlinks=true,
     linkcolor=upsRed,
     citecolor = upsPurp,      
     urlcolor= upsCyan,
     }

\mathchardef\mhyphen="2D

\renewcommand\thesubsection{\thesection.\Alph{subsection}}

\usepackage[
backend=biber,
style=alphabetic,
sorting=nyt
]{biblatex}
\usepackage{csquotes}

\numberwithin{equation}{section}

\newcommand\blfootnote[1]{%
  \begingroup
  \renewcommand\thefootnote{}\footnote{#1}%
  \addtocounter{footnote}{-1}%
  \endgroup
}

\title{{\bf Big mapping class groups are not extremely amenable}}
\author{{\sc Yusen Long}}
\date{}

\begin{document}
\maketitle
\blfootnote{\emph{Date}: July 2023.\\ The author is supported by CSC-202108170018.}

\begin{abstract}
This paper uses the renowned Kechris-Pestov-Todor\v{c}evi\'{c} machinery to show that (big) mapping class groups are not extremely amenable unless the underlying surface is a sphere or a once-punctured sphere, or equivalently when the mapping class group is trivial. The same techniques also show that the pure mapping class groups, as well as compactly supported mapping class groups, of a surface with genus at least one can never be extremely amenable.
\end{abstract}

\vspace{0.25cm}

%\setcounter{tocdepth}{2}
%\tableofcontents

\section{Introduction}
Let $\Sigma$ be an orientable surface. A theorem of Richards \cite{richards1963classification} asserts that it is characterised by the following parameters: genus and the space of ends. Let $g(\Sigma)$ be genus of the surface $\Sigma$ and $n(\Sigma)$ be the cardinality of its ends. The complexity of the surface is defined as $c(\Sigma)\coloneqq 3g(\Sigma)+n(\Sigma)-3$. Such a surface is said to have \emph{finite type} if $c(\Sigma)<\infty$, otherwise it is of \emph{infinite type}.

Define the mapping class group of $\Sigma$ by $\MCG(\Sigma)\coloneqq\mathrm{Homeo}^+(\Sigma)/\{\mathrm{isotopy}\}$, where $\mathrm{Homeo}^+(\Sigma)$ is the group of orientation preserving automorphisms of the surface, equipped with compact-open topology, and $\MCG(\Sigma)$ is carrying the quotient topology. With this topology, the group $\MCG(\Sigma)$ is a \emph{non-archimedean Polish group} (see Proposition \ref{prop_non-arch}), or equivalently it is a closed subgroup of $S_\infty$, the symmetric group of $\mathbb{N}$, with respect to the pointwise convergence topology (see for example \cite[\S 1.5]{becker1996descriptive}). A mapping class group is \emph{big} if the underlying surface is of infinite type.

From a model theoretic aspect, a non-archimedean Polish group can always be realised as the automorphism group of some countable first-order relational structure. Big mapping class groups fall into this category (see Proposition \ref{prop_non-arch}).

Recall that a topological group $G$ is said to have \emph{fixed point on compacta property} or \emph{extremely amenable} if every continuous $G$-action on a compact Hausdorff space admits a fixed point. However, it is worth noticing that, other than the trivial group, no locally compact groups are extremely amenable \cite{veech1977topological}, thus \emph{a fortiori} no discrete ones \cite{ellis1960universal}.

In a celebrated paper \cite{kechris2005fraisse}, Kechris, Pestov and Todor\v{c}evi\'{c} develop a surprising correspondence (\emph{abbrv.} KPT correspondence) between model theory, combinatorics and topological dynamics: if $\mathcal{F}$ is a structure with universe $\mathbb{N}$, then the non-archimedean Polish group $\mathrm{Aut}(\mathcal{F})$ is extremely amenable if and only if the \emph{age} $\mathrm{Age}(\mathcal{F})$ has Ramsey property (see \cite{kechris2005fraisse} for detail). So it is a natural question to ask for a non-archimedean Polish group arising in the study of a geometry object whether or not it is extremely amenable. \emph{are big mapping class groups extremely amenable?} This paper completely answers this question. One shall notice that big mapping class groups are not locally compact \cite[Theorem 4.2]{aramayona2020big}.

\begin{thm}\label{thm1}
Let $\Sigma$ be an orientable surface of finite or infinite topological type. Then $\MCG(\Sigma)$ is not extremely amenable unless $\Sigma$ is a sphere or a once-punctured sphere, in which cases the mapping class groups are trivial.
\end{thm}

Denote by $\End(\Sigma)$ the end space of a surface $\Sigma$. It is a compact space with a natural continuous $\MCG(\Sigma)$-action. In many cases, this action is fixed-point free, which will witness the non extreme amenability of the mapping class group. But this is not always the case. For example, the \emph{Loch Ness monster} surface has infinite genus but only one end and the action of its mapping class group on the end space is trivial. Other non trivial examples are (non self-similar) surfaces with a unique maximal end \cite{mann2022results}.

The proof of Theorem \ref{thm1} mainly relies on the description of extremely amenable non-archimedean Polish groups provided in \cite{kechris2005fraisse} and can be divided into the discussion of two cases, depending on whether the surface is zero-genus or not. When the surface $\Sigma$ has zero genus, the Mann-Rafi ordering \cite{mann2020large} yields two disjoint curves of the same topological type unless $\Sigma$ is a sphere or a once-punctured sphere. In this case, the existence of these two curves further indicates that $\MCG(\Sigma)$ is not extremely amenable. If $g(\Sigma)>0$, then one may also find another pair of curves with the same topological type, not necessarily disjoint, of which the existence also implies the non extreme amenability of $\MCG(\Sigma)$.

Moreover, the proof of Theorem \ref{thm1} remains valid for closed subgroup of $\MCG(\Sigma)$ containing a non-trivial mapping class with finite orbit, including pure mapping class groups $\PMCG(\Sigma)$, \emph{i.e.} the subgroup in $\MCG(\Sigma)$ that fixes pointwise every ends, and the closure of compactly supported mapping class group $\MCG_c(\Sigma)$. 
\begin{thm}\label{thm2}
Let $\Sigma$ be an orientable surface of finite or infinite type. If $G<\MCG(\Sigma)$ is a subgroup containing a mapping class $\phi\in G$ such that for some simple closed curve $c$ on $\Sigma$, the orbit $\{\phi^n(c):n\in\mathbb{Z}\}$ is finite, then $G$ is not extremely amenable. In particular, the groups $\PMCG(\Sigma)$ and $\MCG_c(\Sigma)$ of a surface $\Sigma$ with $g(\Sigma)\geq 1$ are not extremely amenable.
\end{thm}

Note that the non extreme amenability of $\PMCG(\Sigma)$ can also be shown in a way that one constructs a $\PMCG(\Sigma)$-action on the circle $S^1$ without fixed-point from the homomorphisms built in \cite{aramayona2020first}. But this method cannot be adapted to $\MCG_c(\Sigma)$.
\section{Non extreme amenability}
Recall that the \emph{extended mapping class group} $\MCG(\Sigma)^\pm$ is defined similarly as $MCG(\Sigma)$ but using the group of all automorphisms $\mathrm{Homeo}(\Sigma)$ in the stead of the orientation preserving one. The result below is already known following several basic facts about mapping class groups:

\begin{prop}\label{prop_non-arch}
If the surface $\Sigma$ is of complexity at least two, then the mapping class group $\MCG(\Sigma)$ and the extended mapping class group $\MCG^\pm(\Sigma)$ are non-archimedean Polish groups.
\end{prop}
\begin{proof}
Let $\mathcal{C}(\Sigma)$ be the curve graph of the surface $\Sigma$, where vertices are isotopy classes of essential simple closed curves and an edge is attached to two vertices if the two classes have disjoint representatives. It is well-known that for a surface $\Sigma$ with complexity at least two, of finite or infinite type, we have the isomorphism $\MCG^\pm(\Sigma)\simeq\mathrm{Aut}(\mathcal{C}(\Sigma))$ between topological groups \cite{ivanov1997automorphisms,luo2000automorphisms,hernandez2018isomorphisms,bavard2020isomorphisms}. Note that $\mathcal{C}(\Sigma)$ has countably many vertices. The graph $\mathcal{C}(\Sigma)$ is actually a countable relational first-order structure. Such an automorphism group is closed subgroup of $S_\infty$ and thus a non-archimedean Polish group (see for example \cite[Part I, \S9.B(7)]{kechris2012classical}). Hence $\MCG^\pm(\Sigma)$ is a non-archimedean Polish group. To show that $\MCG(\Sigma)$ is also one, it suffices to demonstrate that $\MCG(\Sigma)$ is closed in $\MCG^\pm(\Sigma)$. Indeed, given a convergent sequence of homeomorphisms $\Sigma\to\Sigma$ in $\mathrm{Homeo}(\Sigma)$, if they are all orientation preserving, then so will be their limit.
\end{proof}

To further discuss the extreme amenability of a non-archimedean Polish group, we need to introduce the following notions. Let $G$ be a group acting on a space $X$ and $Y$ be a subspace of $X$. We call the subgroup $G_{(Y)}=\{g\in G: gy=y,\ \forall y\in Y\}$ of $G$ the \textbf{pointwise stabiliser} of $Y$. Similarly, the subgroup $G_Y=\{g\in G: gY=Y\}$ is called the \textbf{setwise stabiliser} of $Y$ in $G$. 

The following observation is a natural consequence after KPT correspondence (see \cite[Proposition 4.3]{kechris2005fraisse}), which can find its root in \cite{glasner2002minimal}.
\begin{lem}\label{lem_stab}
If $G$ is an extremely amenable non-archimedean Polish group acting continuously on the discrete space $\mathbb{N}$, then $G_{(F)}=G_F$ for any finite subset $F\subset \mathbb{N}$.
\end{lem}
\begin{proof}[Sketch of proof]
The continuous action $G$ on $X$ induces a continuous $G$-action on the compact space of linear orders on $X$. Since $G$ is extremely amenable, then we can find a $G$-invariant linear order on $X$. But for any finite set $F\subset X$, there is no non-trivial linear order preserving action. 
\end{proof}

The end space $\End(\Sigma)$ is a totally disconnected, separable, metrisable topological space (and thus a closed subset of Cantor set). Among the ends, there are \emph{non planary ends}, of which the collection is denoted $\End_\infty(\Sigma)$, in the sense that every neighbourhood in $\Sigma$ of a non planary end has at least genus $1$. The non planary ends $\End_\infty(\Sigma)$ form a closed subset of $\End(\Sigma)$.

\begin{rem}
There is a slight \emph{abus de langage} here. The neighbourhood mentioned above resides in the end compactification of $\Sigma$ instead of $\End(\Sigma)$, but its intersection with $\Sigma$ becomes a subsurface in $\Sigma$ and can still be regarded as a ``neighbourhood'' of an end.
\end{rem}

Let $D,D'$ be two subsets of $\End(\Sigma)$. We say that $D$ and $D'$ are \emph{isomorphic} if $D$ is homeomorphic to $D'$ and $D\cap \End_\infty(\Sigma)$ is homeomorphic $D'\cap \End_\infty(\Sigma)$ simultaneously.

Mann and Rafi give a way to order the ends of a surface by their similarity \cite{mann2020large}. Let $x,y\in \End(\Sigma)$ be two ends. Then we write $x\preccurlyeq y$ if every clopen neighbourhood of $y$ contains a clopen subset that is isomorphic to a neighbourhood of $x$. Two ends are said \emph{equivalent} if both $x\preccurlyeq y$ and $y\preccurlyeq x$. They are said \emph{non-comparable} if neither case happens.

It is worth noticing that each compact neighbourhood of an end of $\Sigma$ corresponds to one of its clopen neighbourhoods $D$ in $\End(\Sigma)$ and this compact neighbourhood can be chosen to be the end compactification of a subsurface in $\Sigma$ whose ends are the union of $D$ with an additional isolated point.

Since the group $\mathrm{Homeo}^+(\Sigma)$ acts on the surface $\Sigma$ continuously, this action has a natural continuous extension onto the end compactification of $\Sigma$. By taking the quotient, it is not hard to see that $\MCG(\Sigma)$ acts continuously by homoemorphisms on $\End(\Sigma)$ and $\End_\infty(\Sigma)$ is an invariant subspace of this group action. Moreover, from the definition, this action preserves the ordering given above, namely $gx\preccurlyeq gy$ for any $g\in\MCG(\Sigma)$ if and only if $x\preccurlyeq y$. In particular, if there exists $g\in\MCG(\Sigma)$ such that $gx=y$, then necessarily $x\preccurlyeq y$.

Adopting the terminology from \cite{farb2011primer}, the \textbf{topological type} of a simple closed curve $c$ on $\Sigma$ is $\Sigma\setminus c$. Two curves $c$ and $c'$ on $\Sigma$ are said to \emph{have the same topological type} if there is a homeomorphism between $\Sigma\setminus c$ and $\Sigma \setminus c'$.

\begin{rem}
Since no distinction will be needed here, in the sequel, a curve on the surface can mean either a topological embedding of $S^1$ or its isotopy class, depending on the context. 
\end{rem}

An argument of the renowned \emph{Alexander method} (see for example \cite[\S 2.3]{farb2011primer} and \cite{hernandez2019alexander}) is that given two distinct simple closed curves (up to isotopy) of the same topological type, one can always find a non-trivial mapping class sending one curve to the other. Moreover, this mapping class $g$ can be chosen to have finite order if the surface $\Sigma$ is of finite type. This is not true in general for infinite-type surfaces, but only some special cases will be needed for our purpose here and similar arguments yield the same result for these special cases:
\begin{lem}\label{lem_non_ea_1}
Let $\Sigma$ be a surface with genus $0$. Suppose that there exists two disjoint non-isotopic simple closed curves $c,c'$ on $\Sigma$ with the same topological type. Assume that $c$ cuts $\End(\Sigma)=E\sqcup N$ and $c'$ cuts $\End(\Sigma)=E'\sqcup N'$ in a way that $E\simeq E'$, $E\cap E'=\emptyset$ and $\End(\Sigma)\setminus(E\cup E')\neq \emptyset$. Then there exists a non-trivial mapping class $\phi\in \MCG(\Sigma)$ such that $\phi(c)=c'$ and $\phi^2=\Id$.
\end{lem}
\begin{proof}
Figure \ref{fig:mickey} depicts how such a surface $\Sigma$ looks like. The simplest case is when $E$ and $E'$ reduces to singleton and $\Sigma$ is a third-punctured sphere $S_{0,3}$. Now the desired $\phi\in\MCG(\Sigma)$ is just the symmetry sending $E$ to $E'$. More precisely, cut $\Sigma$ into three parts $S\sqcup S' \sqcup \big(\Sigma\setminus (S\cup S')\big)$ along $c$ and $c'$ so that $\partial S=c$ and $\partial S'=c'$. By Richards' theorem, there is a homeomorphism $\varphi\colon S\to S'$ sending $\partial S$ to $\partial S'$. Take a symmetry $\sigma\in\MCG(\Sigma\setminus(S\cup S'))$ that interchanges the position of $c$ and $c'$: $\sigma$ is the continuous extension by identity of the symmetry $\overline{\sigma}\in\MCG(S_{0,3})$, where $S_{0,3}\subset \Sigma\setminus(S\cup S')$ is a subsurface of finite type with two of its ends being $\partial S=c$ and $\partial S'=c'$. Now $\phi$ is piecewise defined by $\sigma$ on $\Sigma\setminus(S\cup S')$, $\varphi$ on $S$ and $\varphi^{-1}$ on $S'$ (up to isotopy).
\begin{figure}[htbp]
    \centering
    \includegraphics[scale=0.4]{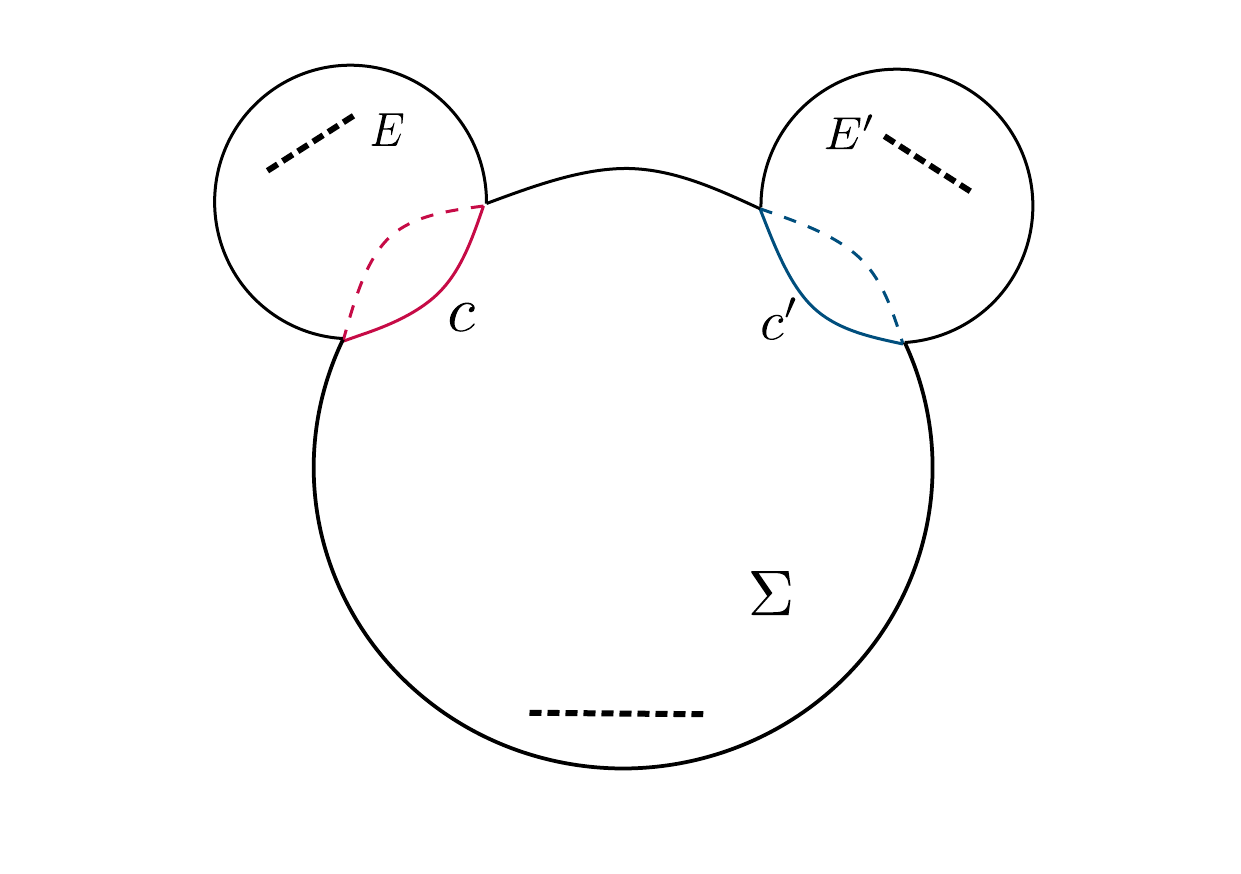}
    \caption{Surface of Lemma \ref{lem_non_ea_1}}
    \label{fig:mickey}
\end{figure}
\end{proof}

Let us first deal with the (non) extreme amenability of $\MCG(\Sigma)$ for surfaces $\Sigma$ with genus $0$ by using the results above.
\begin{lem}\label{lem_non_comparable}
Let $\Sigma$ be an orientable surface with genus $0$ and complexity at least $2$. If $\MCG(\Sigma)$ is extremely amenable, then any two distinct ends in $\End(\Sigma)$ are non-comparable.
\end{lem}
\begin{proof}
Suppose for contradiction that $\MCG(\Sigma)$ is extremely amenable but there exist two distinct $x,y\in \End(\Sigma)$ such that $x\preccurlyeq y$. Since $\End(\Sigma)$ is metrisable and \emph{a fortiori} Hausdorff, we can take a clopen neighbourhood $N$ of $y$ so that $x\notin N$. By definition, inside of $N$, there exists a homeomorphic copy of a clopen neighbourhood of $y$ but excludes $x$. This implies the existence of $x'\in \End(\Sigma)$ which is distinct from $x$ but equivalent to $x$. Moreover, one can take a clopen neighbourhood $D$ of $x$ and a clopen neighbourhood $D'$ of $x'$ in the way that $D\simeq D'$ but $D\cap D'=\emptyset$. We should note that $\End(\Sigma)\setminus D$ is also homeomorphic to $\End(\Sigma)\setminus D'$. Now associate $D$ to a subsurface with boundary $S\subset \Sigma$ as described above so that $\End(S)=D\sqcup\{\ast\}$ and find a subsurface $S'\subset \Sigma$ likewise for $D'$. By Richards' theorem, the curves $\partial S$ and $\partial S'$ have the same topological type and satisfies the hypothesis of Lemma \ref{lem_non_ea_1}. But Lemma \ref{lem_non_ea_1} indicates that there exists a mapping class $\phi\in\MCG(\Sigma)$ such that if $F=\{\partial S,\partial S'\}$, then $\phi\in \MCG(\Sigma)_F\setminus \MCG(\Sigma)_{(F)}$. This yields a contradiction to Lemma \ref{lem_stab}.
\end{proof}

Now one can show the following proposition:
\begin{prop}\label{prop_genus_0}
Let $\Sigma$ be an orientable surface with genus $0$. Then $\MCG(\Sigma)$ is extremely amenable if and only if $\Sigma$ is a sphere or a once-punctured sphere, in which case $\MCG(\Sigma)$ is the trivial group.
\end{prop}
\begin{proof}
If the surface has complexity less than $2$, then it has non-trivial discrete group as mapping class group unless it is a sphere or a once-punctured sphere. So by the virtue of Lemma \ref{lem_non_comparable}, it remains to show that there is no surface with genus $0$, complexity at least $2$ and pairwise non-comparable ends unless it is a sphere or a once-punctured sphere, where the statement is satisfied vacuously. Indeed, let $\Sigma$ be a such surface. Suppose that there exist distinct $x,y\in\End(\Sigma)$ and an element $g\in\MCG(\Sigma)$ so that $gx=y$. Then necessarily $x\preccurlyeq y$ as remarked above, which contradicts to the non-comparing assumption. Hence the action of $\MCG(\Sigma)$ on $\End(\Sigma)$ must be trivial. This implies that $\MCG(\Sigma)=\PMCG(\Sigma)$. By \cite[Theorem 3]{patel2018algebraic}, this means that $\MCG(\Sigma)=\PMCG(\Sigma)$ are simultaneously residually finite and as a result, the surface $\Sigma$ must be of finite type. But as the ends of $\Sigma$ must be all non-comparable, it only happens when $\Sigma$ is a sphere or a once punctured sphere, in which cases $\MCG(\Sigma)$ is trivial.
\end{proof}

Let $S^1_1$ be the surface of genus $1$ with one boundary component and let $S_{1,1}$ be once-punctured torus. Then there is an element $g\in \MCG(S_1^1)$ and a simple closed curve $c$ on $S^1_1$ such that $\{g^n(c):n\in\mathbb{Z}\}$ is finite. Indeed, by the inclusion homomorphism \cite[Theorem 3.18]{farb2011primer}, we have the short sequence
$$1\to \mathbb{Z}\to \MCG(S^1_1)\to \mathrm{SL}_2(\mathbb{Z})\simeq \MCG(S_{1,1})\to 1.$$
Note that 
$$\phi\coloneqq\begin{pmatrix}0 & -1\\ 1& 0\end{pmatrix}\in \mathrm{SL}_2(\mathbb{Z})\simeq \MCG(S_{1,1})$$
is a torsion element and that there is a simple closed curve $c$ on $S_{1,1}$, which can also be regarded as a curve in $S^1_1$, such that $\{\phi^n(c):n\in\mathbb{Z}\}$ is finite. As a result, the orbit of $c$ under the action of any pre-image $g\in \MCG(S^1_1)$ of $\phi\in \MCG(S_{1,1})$ is finite.

With the observation above, one can now prove Theorem \ref{thm1}:

\begin{proof}[Proof of Theorem \ref{thm1}]
If the surface $\Sigma$ has genus $0$, then the extreme amenability of $\MCG(\Sigma)$ is determined by Proposition \ref{prop_genus_0}. For the torus, the mapping class group is not extremely amenable since it is a non-trivial discrete group. Suppose now that the surface $\Sigma$ has complexity at least $2$ and non-zero genus. Then $\Sigma$ contains an essential (sub)surface that is homeomorphic to $S_1^1$. Then we take a $g\in \MCG(S^1_1)$ and a curve $c$ on $S^1_1\subset \Sigma$ so that $\{g^n(c):n\in\mathbb{Z}\}$ is finite. By the virtue of the inclusion homomorphism \cite[Theorem 3.18]{farb2011primer}, one can extend $g$ by identity to an element in $\MCG(\Sigma)$. Now the pointwise stabiliser of $F\coloneqq\{g^n(c):n\in\mathbb{Z}\}$ in $\MCG(\Sigma)$ is different from its setwise stabiliser. Hence by Lemma \ref{lem_stab}, $\MCG(\Sigma)$ is not extremely amenable.
\end{proof}

Finally, there are still some details in Theorem \ref{thm2} that need further clarifications.
\begin{lem}\label{lem_dense}
Let $G$ be a topological group and $H$ be a dense subgroup of $G$. Then $G$ is extremely amenable if and only if $H$ is.
\end{lem}
\begin{proof}
Note that a topological group $G$ is extremely amenable if and only if it admits a $G$-invariant multiplicative mean over $\mathrm{RUCB}(G)$, the space of right-uniformly continuous functions on $G$, or equivalently the $G$-action on its \emph{Samuel compactification} $\sigma G$ has a fixed point (see for example \cite[\S 1.1]{pestov2006dynamics}). But if $H$ is a dense subgroup of $G$, then $\mathrm{RUCB}(H)\simeq \mathrm{RUCB}(G)$ and the canonical continuous $H$-action on the multiplicative means of $\mathrm{RUCB}(G)$ is extended continuously to the canonical $G$-action on it via taking the limit along Cauchy sequences. Conversely, the restriction of a continuous $G$-action on $H$ yields an $H$-action. Hence $H$ and $G$ can only be simultaneously extremely amenable.
\end{proof}
\begin{rem}
More general, for every topological group $G$, a $G$-action on a compact space $X$ is continuous if and only if it is uniformly continuous, \emph{i.e.} uniformly continuous as a map $G\times X\to X$ where $G$ is carrying the right-uniformity. Since Cauchy nets are preserved under uniformly continuous map, every continuous $G$-action on the compact space $X$ can be extended to a continuous action of its completion on $X$.
\end{rem}

Now we are ready to prove Theorem \ref{thm2}.

\begin{proof}[Proof of Theorem \ref{thm2}]
For any subgroup $G<\MCG(\Sigma)$ with complexity $c(\Sigma)>2$ that contains a mapping class $\phi$ such that for some simple closed curve $c$ on the surface, the orbit $F\coloneqq\{\phi^n(c):n\in\mathbb{Z}\}$ is finite, it is clear that $\phi$ belongs to $\overline{G}_{F}$ but not $\overline{G}_{(F)}$. Hence the closure of such subgroup $G$ in $\MCG(\Sigma)$ is not extremely amenable. It follows from Lemma \ref{lem_dense} that $G$ itself is not extremely amenable.
\end{proof}

\section{Perspectives}
Using KPT correspondence, this paper shows that the mapping class groups of all but finitely many orientable surfaces can never be the automorphism group of a countable first-order relational structure $\mathcal{F}$ such that $\mathrm{Age}(\mathcal{F})$ has Ramsey property (see \cite{kechris2005fraisse} for definitions). More generally, we can also ask the following question:
\begin{ques}\label{q1}
Given a countable first-order relational structure $\mathcal{F}$, how can one detect if its automorphism group $\mathrm{Aut}(\mathcal{F})$ can be realised as the (extended) mapping class group of an orientable surface?
\end{ques}

In a recent paper, Disarlo, Koberda and Gonz\'{a}lez \cite{disarlo2023model} establish a model theoretic connection between the mapping class groups and the curve graph of non-sporadic finite-type surfaces, which is motivated by Ivanov's metaconjecture \cite{ivanov2006fifteen}. Following their ideas, another way to ask Question \ref{q1} is the following:
\begin{ques}
Given a graph on countable vertices, how can one detect if it is the curve graph of an orientable surface?
\end{ques}

Another notion that is closely related to the extreme amenability is the amenability of a topological group, \emph{i.e.} every continuous group action on a compact Hausdorff space admits an invariant probability measure over the space. Although the non extreme amenability is already clear, it remains unknown if there are amenable big mapping class groups. For every but finitely many finite-type orientable surfaces, the mapping class group is non-amenable because as a discrete group, it contains a non-abelian free subgroup on two generators. If the surface is of infinite type, the amenability of its mapping class group is less clear. For the surfaces $\Sigma$ of infinite type having a non-displaceable subsurface $S$ of finite type, one can construct a \emph{blown-up projection complex} from the curve graphs of the $\MCG(\Sigma)$-orbit of $S$, see for example \cite{horbez_qing_rafi_2022,domat2022big}. Equipped with the combinatorial metric, the blown-up projection complex is a separable geodesic Gromov-hyperbolic space on which the mapping class group $\MCG(\Sigma)$ acts continuously by isometries and the $\MCG(\Sigma)$-action is of general type (or sometimes non-elementary). However, an amenable group can never have a continuous action of general type on a separable geodesic Gromov-hyperbolic space by isometries. This implies that these big mapping class groups are not amenable.

\section*{Acknowledgements}
The author thanks George Domat, Sergio Domingo, Bruno Duchesne and Jes\'{u}s Hernandez Hernandez for helpful discussions.

\printbibliography

@article{kechris2005fraisse,
  title={Fra{\"i}ss{\'e} Limits, Ramsey Theory, and topological dynamics of automorphism groups},
  author={Kechris, Alexander S and Pestov, Vladimir G and Todorcevic, Stevo},
  journal={Geometric and Functional Analysis},
  volume={1},
  number={15},
  pages={106--189},
  year={2005}
}

@article{veech1977topological,
  title={Topological dynamics},
  author={Veech, William A},
  journal={Bulletin of the American Mathematical Society},
  volume={83},
  number={5},
  pages={775--830},
  year={1977}
}

@article{ellis1960universal,
  title={Universal minimal sets},
  author={Ellis, Robert},
  journal={Proceedings of the American Mathematical Society},
  volume={11},
  number={4},
  pages={540--543},
  year={1960}
}

@book{becker1996descriptive,
  title={The descriptive set theory of Polish group actions},
  author={Becker, Howard and Kechris, Alexander S},
  volume={232},
  year={1996},
  publisher={Cambridge University Press}
}

@article{ivanov1997automorphisms,
  title={Automorphisms of complexes of curves and of Teichmuller spaces},
  author={Ivanov, Nikolai V},
  journal={International Mathematics Research Notices},
  volume={1997},
  number={14},
  pages={651--666},
  year={1997},
  publisher={Duke University Press}
}

@article{luo2000automorphisms,
  title={Automorphisms of the complex of curves},
  author={Luo, Feng},
  journal={Topology},
  volume={39},
  number={2},
  pages={283--298},
  year={2000},
  publisher={Pergamon}
}

@article{bavard2020isomorphisms,
  title={Isomorphisms between big mapping class groups},
  author={Bavard, Juliette and Dowdall, Spencer and Rafi, Kasra},
  journal={International Mathematics Research Notices},
  volume={2020},
  number={10},
  pages={3084--3099},
  year={2020},
  publisher={Oxford University Press}
}

@book{kechris2012classical,
  title={Classical descriptive set theory},
  author={Kechris, Alexander},
  volume={156},
  year={2012},
  publisher={Springer Science \& Business Media}
}

@misc{mann2020large,
      title={Large scale geometry of big mapping class groups}, 
      author={Kathryn Mann and Kasra Rafi},
      year={2020},
      eprint={1912.10914},
      archivePrefix={arXiv},
      primaryClass={math.GT}
}

@book{farb2011primer,
  title={A primer on mapping class groups (pms-49)},
  author={Farb, Benson and Margalit, Dan},
  volume={41},
  year={2011},
  publisher={Princeton university press}
}

@article{hernandez2019alexander,
  title={The Alexander method for infinite-type surfaces},
  author={Hern{\'a}ndez, Jes{\'u}s Hern{\'a}ndez and Morales, Israel and Valdez, Ferr{\'a}n},
  journal={Michigan Mathematical Journal},
  volume={68},
  number={4},
  pages={743--753},
  year={2019},
  publisher={University of Michigan, Department of Mathematics}
}

@article{richards1963classification,
  title={On the classification of noncompact surfaces},
  author={Richards, Ian},
  journal={Transactions of the American Mathematical Society},
  volume={106},
  number={2},
  pages={259--269},
  year={1963},
  publisher={JSTOR}
}

@inproceedings{aramayona2020big,
  title={Big mapping class groups: an overview},
  author={Aramayona, Javier and Vlamis, Nicholas G},
  booktitle={In the tradition of Thurston: geometry and topology},
  pages={459--496},
  year={2020},
  publisher={Springer}
}

@misc{disarlo2023model,
      title={The model theory of the curve graph}, 
      author={Valentina Disarlo and Thomas Koberda and J. de la Nuez González},
      year={2023},
      eprint={2008.10490},
      archivePrefix={arXiv},
      primaryClass={math.GT}
}

@inproceedings{ivanov2006fifteen,
    title={Fifteen problems about the mapping class groups},
    author={Ivanov, Nicolai V},
    booktitle={Proceedings of Symposia in Pure Mathematics},
    year={2006},
    volume={74},
    pages={71--80}
}

@article{aramayona2020first,
  title={The first integral cohomology of pure mapping class groups},
  author={Aramayona, Javier and Patel, Priyam and Vlamis, Nicholas G},
  journal={International Mathematics Research Notices},
  volume={2020},
  number={22},
  pages={8973--8996},
  year={2020},
  publisher={Oxford University Press}
}

@book{pestov2006dynamics,
    author = {Pestov, Vladimir},
    title = {Dynamics of Infinite-dimensional Groups: The Ramsey–Dvoretzky-Milman Phenomenon},
    publisher = {American Mathematical Society},
    series = {University Lecture Series},
    volume ={40},
    year = {2006}
}

@article{horbez_qing_rafi_2022,
    title={Big mapping class groups with hyperbolic actions: classification and applications},
    volume={21}, 
    number={6}, 
    journal={Journal of the Institute of Mathematics of Jussieu},
    publisher={Cambridge University Press},
    author={Horbez, Camille and Qing, Yulan and Rafi, Kasra},
    year={2022},
    pages={2173–2204}
}

@article{domat2022big,
  title={Big pure mapping class groups are never perfect},
  author={Domat, George and Dickmann, Ryan},
  journal={Mathematical Research Letters},
  volume={29},
  number={3},
  pages={691--726},
  year={2022},
  publisher={International Press of Boston}
}

@misc{mann2022results,
      title={Two results on end spaces of infinite type surfaces}, 
      author={Kathryn Mann and Kasra Rafi},
      year={2022},
      eprint={2201.07690},
      archivePrefix={arXiv},
      primaryClass={math.GT}
}

@article{glasner2002minimal,
  title={Minimal actions of the group of permutations of the integers},
  author={Glasner, Eli and Weiss, Benjamin},
  journal={Geometric \& Functional Analysis GAFA},
  volume={12},
  number={5},
  pages={964--988},
  year={2002},
  publisher={Springer}
}

@article{patel2018algebraic,
  title={Algebraic and topological properties of big mapping class groups},
  author={Patel, Priyam and Vlamis, Nicholas},
  journal={Algebraic \& Geometric Topology},
  volume={18},
  number={7},
  pages={4109--4142},
  year={2018},
  publisher={Mathematical Sciences Publishers}
}

@article{hernandez2018isomorphisms,
author = {Jesus Hernandez Hernandez and Israel Morales and Ferran Valdez},
title = {{Isomorphisms between curve graphs of infinite-type surfaces are geometric}},
volume = {48},
journal = {Rocky Mountain Journal of Mathematics},
number = {6},
publisher = {Rocky Mountain Mathematics Consortium},
pages = {1887 -- 1904},
year = {2018}
}

\noindent{\sc Université Paris-Saclay, Laboratoire de Mathématique d'Orsay, 91405, Orsay, France}

\noindent{\it Email address:} {\tt \href{mailto:yusen.long@universite-paris-saclay.fr}{yusen.long@universite-paris-saclay.fr}}
\end{document}